\documentclass[12pt,reqno]{amsart}
\usepackage{fullpage,  amssymb, amsthm, amsmath, amsfonts, times}
\usepackage[english]{babel}

\usepackage{enumitem}
\usepackage{mathrsfs}
\usepackage{mathtools}
\usepackage{centernot}
\usepackage{xfrac}
\usepackage{eufrak}
\usepackage{stmaryrd}
\usepackage{bm}
\usepackage[all,cmtip]{xy}
\usepackage[usenames,dvipsnames]{xcolor}
\usepackage[margin=1.0 in]{geometry}
\usepackage{float}
\usepackage{hyperref}
\hypersetup{colorlinks=true, citecolor=blue, linkcolor=blue, urlcolor=blue, pdfstartview=FitH, pdfauthor=CurranHeycock, pdftitle=}
\usepackage{tikz-cd}
\usetikzlibrary{positioning,arrows}
\usetikzlibrary{calc}
\usepackage[utf8]{inputenc}

\theoremstyle{plain} 
\newtheorem{theorem}{Theorem}[section]
\newtheorem*{theorem*}{Theorem}

\newtheorem{lemma}[theorem]{Lemma}

\theoremstyle{remark}
\newtheorem{exr}{Exercise}

\theoremstyle{definition}

\renewcommand{\mod}[1]{{\ifmmode\text{\rm\ (mod~$#1$)}\else\discretionary{}{}{\hbox{ }}\rm(mod~$#1$)\fi}}

\newcommand{\ceil}[1]{\left\lceil#1\right\rceil}

\renewcommand{\Im}{\textup{Im }}
\renewcommand{\Re}{\textup{Re }}

\renewcommand{\b}{\beta}

\renewcommand{\t}{\theta}

\newcommand{\Z}{\mathbb Z}

\newcommand{\R}{\mathbb R}
\newcommand{\C}{\mathbb C}

\newcommand{\B}{\mathcal B}
\newcommand{\p}{\mathcal P}

\newcommand{\n}{\mathcal N}
\newcommand{\I}{\mathcal I}
\newcommand{\G}{\mathcal G}

\newcommand{\ignore}[1]{}

\newcommand{\half}{\tfrac{1}{2}} 

\newcommand{\fP}{\mathcal{P}}
\newcommand{\fN}{\mathcal{N}}
\newcommand{\fG}{\mathcal{G}}
\newcommand{\fB}{\mathcal{B}}

\allowdisplaybreaks 

\title{Sharp bounds for Joint moments of the Riemann zeta function}

\author{Michael J. Curran}
\address{Mathematical Institute, University of Oxford, Oxford, OX2 6GG, UK.}
\email{michael.curran@maths.ox.ac.uk}

\author{André Heycock}
\address{Department of Mathematics, University of Manchester, Manchester, M13 9PL, UK.}
\email{andre.heycock@postgrad.manchester.ac.uk}

\begin{document}
\nocite{*}

\maketitle

\begin{abstract}
In \cite{curran-1}, the first author obtained conjecturally sharp upper bounds for the joint moments of the $(2k-2h)^{\text{th}}$ power of the Riemann zeta function with the $2h^{\text{th}}$ power of its derivative on the critical line in the range $1\leq k \leq 2$, $0 \leq h \leq 1$. Unconditionally, we extend these upper bounds to all $0 \leq h\leq k \leq 2$, and obtain lower bounds for all $0\leq h \leq k+1/2$. Assuming the Riemann hypothesis, we give sharp bounds for all $0\leq h \leq k$. We also prove upper bounds of the conjectured order for more general joint moments of zeta with its higher derivatives.
\end{abstract}


\section{Introduction}

This paper is concerned with the joint moments of the Riemann zeta function with its derivative on the critical line and other natural variants of these moments.
Work of Hall \cite{hall}, Hughes \cite{Hughes}, and   Keating and Snaith \cite{KS2, KS1} has led to the conjecture that for $k > -\tfrac{1}{2}$ and $-\tfrac{1}{2} < h < k + \tfrac{1}{2}$,
\begin{equation} \label{eqn:JMConjecture}
\int_{T}^{2T} |\zeta(\tfrac{1}{2} + i t)|^{2k - 2 h} \left|\zeta'(\tfrac{1}{2} + i t)\right|^{2h} \ dt \sim C(k,h) T (\log T)^{k^2 + 2h} 
\end{equation}
as $T \rightarrow \infty$ for certain constants $C(k,h)$.
These constants can be expressed as the product of a natural number theoretic factor and a factor coming from conjectural connections of the Riemann zeta function with random matrix theory.
The arithmetic factor is given by a product over primes, while the random matrix factor can be expressed using Painlevé trancendents.
For finite $N$, the random matrix factor can be expressed using the solution to a Painlevé V differential equation \cite{P5}, while the random matrix factor in the limit $N\rightarrow \infty$ is related to the solution of a Painlevé III differential equation \cite{AKW, BBBCPRS, P5, FW}.
This asymptotic (\ref{eqn:JMConjecture}) is currently only known to hold for integer $h , k$ satisfying $0 \leq h \leq k \leq 2$ due to work of Ingham \cite{Ingham} and Conrey \cite{Conrey}.

Since the asymptotic (\ref{eqn:JMConjecture}) seems out of reach for all remaining values of $h$ and $k$, it is natural to ask if we can at least obtain upper bounds of the right order.
The work of Heap, Radziwiłł, and Soundararajan \cite{heap-radziwill-soundararajan} shows that we have upper bounds of the expected order for (\ref{eqn:JMConjecture}) for $h = 0, k \leq 2$. 
In fact, lower bounds of the conjectured order are also known for all $k \geq 0$ when $h = 0$ due to work of Heap and Soundararajan \cite{heap-soundararajan}. 
By using Conrey's asymptotic for the fourth moment of the derivative of zeta \cite{Conrey} along with Hölder's inequality, one also sees that an upper bound of the expected order for (\ref{eqn:JMConjecture}) holds for $k = 2$ and $h \leq 2$.
Finally, the first author recently showed \cite{curran-1} that the an upper bound of the expected order for (\ref{eqn:JMConjecture}) holds in the larger range where $0\leq h\leq 1$ and $1\leq k \leq 2$.
The paper \cite{curran-1} uses the work of Heap, Radziwiłł, and Soundararajan \cite{arguin-ouimet-radziwill} along with the observation that if we have an upper bound of the expected order in (\ref{eqn:JMConjecture}) for a given pair $k, h$, then we also have an upper bound of the expected order for the pair $k, h'$ for any $0\leq h' \leq h$.
This latter observation is a simple consequence of Hölder's inequality and will be quite useful in this paper.\\

Our primary aim is to prove the upper bound (\ref{eqn:JMConjecture}) for all $0 \leq h \leq k \leq 2$. 
The method of proof extends straightforwardly to joint moments with higher derivatives. Assuming the Riemann hypothesis, the bound extends to all $0 \leq h \leq k$.
\begin{theorem}\label{thm:mainRH}
Let $n\in\Z_{>0}$, $0\leq k \leq 2$, and $h_j \geq 0$ for $1\leq j \leq n$, with $h:=h_1 + \cdots + h_n \leq k$. Then
\begin{equation}\label{eq:generalMM}
        \int_{T}^{2T}|\zeta(\half+it)|^{2k-2h}\prod_{j=1}^n|\zeta^{(j)}(\half+it)|^{2h_j} \ dt\ll_{k,n,h_1,\dots,h_n} T(\log T)^{k^2 + 2\sum_{j=1}^n j h_j},
    \end{equation}
    unconditionally for $k\leq 2$, and under the assumption of the Riemann hypothesis for $k>2$.
\end{theorem}
\noindent
The upper bound in Theorem \ref{thm:mainRH} is the same order of magnitude one would predict using random matrix theory. The case of integral joint moments of a product of two derivatives of characteristic random unitary matrices has been extensively studied by Keating and Wei \cite{keating-wei-1,keating-wei-2}.\\

When $n=1$, Theorem \ref{thm:mainRH} 
provides an upper bound of the expected order in (\ref{eqn:JMConjecture}) for all $0 \leq h \leq k$, assuming the Riemann hypothesis. We also obtain unconditional lower bounds for all $k \geq 0$ with $0\leq h < k + \tfrac{1}{2}$, so the upper bounds in (\ref{eqn:JMConjecture}) are of the correct order. 
\begin{theorem} \label{thm:lowerJM}
Let $k \geq 0$ and $0\leq h < k + \tfrac{1}{2}$. Then
\[
\int_T^{2T} |\zeta(\tfrac{1}{2} + i t)|^{2k - 2h} |\zeta'(\tfrac{1}{2} + i t)|^{2h} \ d t \gg_{k,h} T (\log T)^{k^2 + 2 h}.
\]
\end{theorem}
\noindent
Lower bounds for joint moments involving higher derivatives will be considered in future work of the second author.

We will begin by showing that Theorem \ref{thm:lowerJM} is a short consequence of the lower bounds for the $2k^\text{th}$ moments of zeta for $k \geq 0$ obtained by Heap and Soundararajan \cite{heap-soundararajan}.
The argument we give is essentially the same argument used in Conrey's \cite{Conrey} proof of the explicit lower bound
\[
\int_T^{2T} |\zeta(\tfrac{1}{2} + i t)|^{3} |\zeta'(\tfrac{1}{2} + i t)| \ d t \geq (1 + o(1)) \frac{173}{672\pi^2} T (\log T)^5.
\]
Since obtaining good implicit constants is not our concern in this paper, a slightly simpler version of the argument in \cite{Conrey} will suffice.

We will then move onto proving the upper bounds.
Using Hölder's inequality, we will demonstrate that it suffices to prove Theorem \ref{thm:mainRH} in the case of a single moment,
\begin{equation*}
    \I_T^{(n)}(k):=\int_T^{2T}|\zeta^{(n)}(\half+it)|^{2k} dt
\end{equation*}
where $n\in\Z_{\geq 0}$, $k\geq 0$.
The proof then splits into two partially overlapping cases. If $k \geq 1/2$, one can represent the derivative of zeta by a contour integral and can use Hölder's or Jensen's inequality to essentially bound (\ref{eqn:JMConjecture}) by $(\log T)^{2nk}$ times the $2k^\text{th}$ moment of zeta.
The argument here is similar to one in work of Milinovich \cite{Mili} on moments of derivatives of zeta, and the proof fails for small $k$ because $x \mapsto x^{2k}$ is not convex when $k < 1/2$.

Our argument in the range $0\leq k \leq 2$ is similar to work of Heap, Radziwiłł, and Soundararajan \cite{heap-radziwill-soundararajan}. However, we initially apply Holder's inequality, instead of Young's inequality, to obtain a sharp bound for $|\zeta^{(n)}(\half+it)|^{2k}$ averaged over a set $S\subseteq[T,2T]$ in terms of a product of the mean value of $|\zeta^{(n)}(\half+it)|^4$ times the exponential of a Dirichlet polynomial and the mean value of the exponential of another Dirichlet polynomial. Splitting $[T,2T]$ into subsets based on the longest permissible Dirichlet polynomial approximation to $\log|\zeta(\half+it)|$, one then bounds the exponentials of Dirichlet polynomials by truncating the exponential to a Taylor polynomial. This reduces the theorem to a variety of twisted moment calculations. For Dirichlet polynomials, these have already been essentially computed in \cite{heap-radziwill-soundararajan}. For twisted fourth moments of higher derivatives of the Riemann zeta function, we adapt the method of \cite[Lem. 2]{curran-1}.
 
\section*{Acknowledgements}
\noindent
Part of this work was written during a visit by the first author to the University of Manchester. He would like to thank Hung Bui and that institution for their hospitality and support, as well as his advisor Jon Keating for useful comments. The second author would like to thank his advisor, Hung Bui, for helpful discussions. Finally both authors would like to thank Winston Heap for comments on an early draft of the paper, and to thank Soundararajan for bringing the argument in \cite{Conrey} to their attention.
\section{Lower Bounds: Proof of Theorem \ref{thm:lowerJM}}

We begin by writing
\[
\int_T^{2T}|\zeta(\tfrac{1}{2} + i t)|^{2k - 2h} |\zeta'(\tfrac{1}{2} + i t)|^{2h} \ d t =\int_T^{2T}|\zeta(\tfrac{1}{2} + i t)|^{2k} \left(\left|\Re \frac{\zeta'}{\zeta}(\tfrac{1}{2}+ i t)\right|^2+ \left|\Im \frac{\zeta'}{\zeta}(\tfrac{1}{2}+ i t)\right|^2\right)^{h} \ d t.
\]
It is a standard consequence of Stirling's approximation and the functional equation that for large $T$,
\[
\Re \frac{\zeta'}{\zeta}(\tfrac{1}{2}+ i t) = -\frac{1}{2}\log\frac{t}{2\pi} + O(|t|^{-1}).
\]
Therefore for large $T$,
\[
\int_T^{2T}|\zeta(\tfrac{1}{2} + i t)|^{2k - 2h} |\zeta'(\tfrac{1}{2} + i t)|^{2h} \ d t \gg_{k,h} (\log T)^{2h} \int_T^{2T}|\zeta(\tfrac{1}{2} + i t)|^{2k}  \ d t.
\]
Theorem \ref{thm:lowerJM} now follows after an application of the main theorem in \cite{heap-soundararajan}, that the integral on the right-hand side above is $\gg_{k} T(\log T)^{k^2}$. \qed
\section{Upper Bounds: Proof of Theorem \ref{thm:mainRH} for $0\leq k \leq 2$}

\subsection{Initial reduction}
First apply H\"older's inequality to separate out each derivative:
    \begin{multline}\label{eq:RH_Thm_FirstHolder}
        \int_{T}^{2T}|\zeta(\half+it)|^{2k-2h}\prod_{j=1}^n|\zeta^{(j)}(\half+it)|^{2h_j} \ dt \\
         \leq \left(\int_{T}^{2T}|\zeta(\half+it)|^{2k}\ dt\right)^{1-h/k} \times \prod_{j=1}^n \left(\int_T^{2T} |\zeta^{(j)}(\half+it)|^{2k} \ dt \right)^{h_j/k}.
    \end{multline}
Theorem \ref{thm:mainRH} follows once we show for all $n\in\Z_{>0}$ that
    \begin{equation}\label{eq:RH_Thm.1}
        \I^{(n)}_T(k):=\int_{T}^{2T} |\zeta^{(n)}(\half+it)|^{2k}\ dt\ll_{k,n} T(\log T)^{k^2+2nk},
    \end{equation}
unconditionally for $k\leq 2$ and on RH for $k>2$. Going forward we will omit the subscripts in our asymptotic notation, so the implicit constants may depend on $k$ and $n$. 
Note that the case $n=0$ is proved in \cite{heap-radziwill-soundararajan} unconditionally for $k\leq 2$ and \cite{harper} on RH. The remainder of this section completes the proof when $0\leq k \leq 2$.

\subsection{Notation} 
Throughout, our parameters will depend on the value of $k$.
Let $\log_j$ denote the $j$-fold iterated logarithm, set $T_0 = 1$,
\[
T_j := \exp\left(\frac{\log T}{(\log_{j+1} T)^2}\right)
\]
for $j \geq 1$, and 
\[
J := \max\left\{j : \log_j T \geq 10^4 \right\}.
\]
For $1\leq j \leq J$ set
\[
\p_j(s) := \sum_{T_{j-1} < p \leq T_{j} } \frac{1}{p^s}, \quad P_j := \sum_{T_{j-1} < p \leq T_{j} } \frac{1}{p}.
\]
Next define the Dirichlet polynomials
\[
\n_j(s;\b) = \sum_{\substack{p \mid n \Rightarrow p\in (T_{j-1},T_j] \\ \Omega(n) \leq 10 K_j}} \frac{\b^{\Omega(n)} g(n)}{n^s},
\]
where $|\b| \leq 2$,
\[
K_j := 50 P_j, \  g(n) = \prod_{p^r \| n}\tfrac{1}{r!},\ \text{ and } \Omega(n) = \sum_{p^r \| n} r.
\]
Note $\n_j(s;\b)$ is a Dirichlet polynomial of length $\leq T_j^{500 P_j}$, so $\prod_{j \leq J}\n_j(s;\b)$ has length at most $T_1^{500 P_1} T_2^{500 P_2} \cdots T_J^{500  P_J} \leq T^{1/10}$.

We will decompose the interval $[T,2T]$ depending on the sizes of the $\p_j(s)$. Define the good set
\[
\G := \left\{t\in [T,2T] : |\p_j(\tfrac{1}{2} + i t)| \leq K_j \text{ for all } j \leq J \right\},
\]
and the bad sets
\[
\B_r := \{t\in [T,2T]: |\p_j (\tfrac{1}{2}+ i t)| \leq K_j \text{ for all } 1\leq j < r \text{ but } |\p_{r} (\tfrac{1}{2} + i t)|  > K_{r} \}.
\]
A similar decomposition occurs in the work of Harper \cite{harper} and also occurs implicitly in the work of Radziwiłł and Soundararajan \cite{radziwill-soundararajan} and related papers \cite{curran-1, heap-radziwill-soundararajan, heap-soundararajan}.
In our argument, it will be important to keep careful track of which values $t$ are good or bad because we will use an interpolation inequality separately on each set of our partition of $[T,2T]$.
Throughout we will use the following lemmata. The second is essentially contained in Proposition 1 of \cite{heap-radziwill-soundararajan}, while the first is an analogue of their use of Young's inequality there.

\begin{lemma}\label{lem:interpol}
For $0\leq \b \leq 2$, $r \leq J +1$, and $S\subseteq[T,2T]$  measurable
\begin{equation*}
    \begin{split}
        \int_S |\zeta(\half+it)|^{2\b}\ dt &\ll \left(\int_S |\zeta(\half+it)|^4 \exp\left((2\b-4) \sum_{j < r} \Re \p_j(\half+it)\right)\ dt \right)^{\beta/2}\\
        &\qquad\times\left(\int_S \exp\left(2\b \sum_{j < r} \Re \p_j(\half+it)\right)\ dt\right)^{1-\beta/2}.
    \end{split}
\end{equation*}
\end{lemma}
\begin{proof}
    For convenience let $V(t) = \sum_{j<r}\Re \fP_j(\half+it)$, and write
    \begin{align*}
        \int_S |\zeta(\half+it)|^{2\beta} \ dt = \int_S |\zeta(\half+it)|^{2\beta}\exp\left(\beta(\beta-2)V(t)\right)\cdot\exp\left(\beta(2-\beta)V(t)\right)\ dt.
    \end{align*}
    The claim follows by applying H\"older's inequality with exponent $p=2/\beta$ for the first factor $|\zeta(\half+it)|^{2\beta}\exp(\beta(\beta-2)V(t))$ and $q=2/(2-\beta)$ for the second factor.
\end{proof}

\begin{lemma}\label{lem:taylorExp}
If $|\b| \leq 2$, $j \leq J$ and $|\p_j(s) | \leq K_j$, then
\[
\exp\left(2\b \Re \p_j(s)\right) = (1 + O(e^{-P_j})) |\n_j(s; \b)|^2.
\]
\end{lemma}
\begin{proof}
By Taylor expansion of the exponential up to the $10 K_j^\text{th}$ term we find
\[
\left|\exp(\b \p_j(s)) - \n_j(s;\b)\right| \leq e^{-10  K_j}
\]
since $|\b \p_j(s)| \leq  2 K_j$.
Therefore
\[
\exp\left(2\b \Re \p_j(s)\right) = |\exp(\b \p_j(s))|^2 = |\n_j(s;\b) + \theta e^{-10 K_j}|^2
\]
where $\t$ denotes a quantity depending on $s$ with $|\t|\leq 1$.
Finally by assumption $|\exp(\b \p_j(s))| \geq e^{- K_j}$ and therefore also $|\n_j(s;\b)| \geq e^{- K_j/2}$, so we may turn the additive error into a multiplicative error.
\end{proof}

\subsection{Reduction to twisted moments}

Theorem \ref{thm:mainRH} will follow from the bounds
\begin{equation}\label{eq:zeta_UB_good}
    \int_{\fG} |\zeta^{(n)}(\half+it)|^{2k} \ dt \ll T(\log T)^{k^2+2nk}
\end{equation}
and
\begin{equation}\label{eq:zeta_UB_bad}
    \int_{\fB_r} |\zeta^{(n)}(\half+it)|^{2k} \ dt \ll T(\log T)^{k^2+2nk}e^{-10P_r}
\end{equation}
after summing over $1\leq r \leq J$.

For $t\in\fG$, applying Lemma \ref{lem:interpol} and then Lemma \ref{lem:taylorExp} with $r=J+1$ yields 
\begin{equation}\label{eq:UB_twisted_good}
    \begin{split}
    \int_\fG|\zeta^{(n)}(\half+it)|^{2k}\ d t
    &\ll \left(\int_{\fG} |\zeta^{(n)}(\half+it)|^4\prod_{j < J+1}|\fN_{j}(\half+it,k-2)|^2\ d t\right)^{\frac{k}{2}}\\
    &\qquad\times\left( \int_{\fG} \prod_{j < J+1}|\fN_{j}(\half+it,k)|^2\ d t\right)^{1-\frac{k}{2}}.
    \end{split}
\end{equation}
Here we have used that $\prod_{j < J+1} (1 + O(e^{-P_j})) = O(1)$.
The range of integration for each integral on the right-hand side can then be extended to $[T,2T]$ since both integrands are non-negative.

The moments over the sets $\B_r$ for $1\leq r \leq J$ are handled in a completely analogous manner --- the only differences are that we apply Lemma \ref{lem:taylorExp} with the sum truncated at $r$ instead of $J+1$ and that we multiply the integrand by $|\p_r(s)/(50 P_r)|^{2\ceil{50P_r}}$, which majorises the indicator function of $\B_r$. Finally, we extend the range of integration to all of $[T,2T]$ and obtain
\begin{equation}\label{eq:UB_twisted_bad}
    \begin{split}
    \int_{\fB_r}|\zeta^{(n)}(\half+it)|^{2k} \ d t
    & \ll \left(\int_{T}^{2T} |\zeta^{(n)}(\half+it)|^4\prod_{j < r}\left|\fN_{j}(\half+it,k-2)\right|^2\left|\frac{\fP_r(\half+it)}{50P_r}\right|^{2\lceil 50 P_r \rceil} \ d t\right)^{\frac{k}{2}}\\
    &\qquad \times\left(\int_{T}^{2T} \prod_{j < r}\left|\fN_{j}(\half+it,k)\right|^2\left|\frac{\fP_r(\half+it)}{50P_r}\right|^{2\lceil 50 P_r \rceil} \ d t\right)^{1-\frac{k}{2}}.
    \end{split}
\end{equation}
Thus the problem is reduced to a handful of twisted moment or mean value calculations.

\subsection{Twisted fourth moment bound for $\zeta^{(n)}(\half+it)$}

To bound the first integrands on the right-hand sides of \eqref{eq:UB_twisted_good} and \eqref{eq:UB_twisted_bad}, we will require a bound for the twisted fourth moment of $\zeta(\half+it)$.
\begin{lemma}\label{lem:general_twisted_4th}
    Let $n\in\Z_{>0}$, and let $A(s):=\sum_{\ell\leq L}a_\ell \ell^{-s}$ with $a_\ell\in\C$ be a Dirichlet polynomial of length $L \ll  T^{\theta}$, $\theta<1/4$. Then
    \begin{equation*}
        \int_T^{2T}|\zeta^{(n)}(\half+it)|^4|A(\half+it)|^2\ d t\ll T(\log T)^{4 + 4n}\max_{|z_j|=\frac{3^j}{\log T}}|G(z_1,z_2,z_3,z_4)|,
    \end{equation*}
    where
    \begin{align*}
        G(z_1,z_2,z_3,z_4)&:=\sum_{h,k\leq L}\frac{a_h\overline{a_k}}{[h,k]}B_{z_1,z_2,z_3,z_4}\left(\frac{h}{(h,k)}\right)B_{z_3,z_4,z_1,z_2}\left(\frac{k}{(h,k)}\right),\\
        A(z_1,\dots,z_4) :=& \frac{\zeta(1+z_1+z_3)\zeta(1+z_1+z_4)\zeta(1+z_2+z_3)\zeta(1+z_2+z_4)}{\zeta(2+z_1+z_2+z_3+z_4)},\\
        B_{z_1,\dots,z_4}&(n) :=\prod_{p^m\| n} \left(\frac{\sum_{j\geq 0} p^{-j}\sigma_{z_1,z_2}(p^{j+m})\sigma_{z_3,z_4}(p^j)}{\sum_{j\geq 0}p^{-j}\sigma_{z_1,z_2}(p^j)\sigma_{z_3,z_4}(p^j)}\right),\\
        &\quad \Delta(z_1,\dots,z_4) := \prod_{1\leq j<\ell\leq 4}(z_j-z_\ell).
    \end{align*}
\end{lemma}
\begin{proof}
    This adapts the proof of \cite[Lem. 2]{curran-1}. One initially has
    \begin{equation}\label{eq:4th_moment_lemma_integral}
        \begin{split}
        I_T(\alpha_1,\dots,\alpha_4)&=\int_T^{2T}\zeta(\half+it+\alpha_1)\zeta(\half+it+\alpha_2)\zeta(\half-it+\alpha_3)\zeta(\half-it+\alpha_4)|A(\half+it)|^2\ d t\\
        &= \frac{1}{(2\pi i)^4}\oint\oint\oint\oint\int_{\R} 
         A(z_1,z_2,-z_3,-z_4)G(z_1,z_2,-z_3,-z_4)\Delta(z_1,z_2,z_3,z_4)^2\\
         &\qquad\qquad\qquad\times\left(\prod_{j=1}^4\frac{\tau^{-\alpha_j/2}}{(z_1-\alpha_j)(z_2-\alpha_j)(z_3+\alpha_j)(z_4+\alpha_j)}\right)\\
         &\qquad\qquad\qquad\qquad\times\tau^{\frac{z_1+z_2-z_3-z_4}{2}}\phi(t/T)\ d t \ dz_1 \ dz_2 \ dz_3 \ dz_4 + O(T^{1-\delta}),
        \end{split}
    \end{equation}
    where $\tau=t/(2\pi)$, $\delta$ is a small positive quantity, the contour for each $z_j$ is an anticlockwise circle $|z_j|=3^{j}/\log T$, and the error term is uniform in $|\alpha_j|\ll 1/\log T$.
    
    By the same argument as in \cite[Lem. 2]{curran-1}, we may differentiate $n$ times with respect to each $\alpha_j$, under the integral and evaluate at $\alpha_1=\dots=\alpha_4=0$, so that
     \begin{equation*}
        \begin{split}
            \frac{\partial^{4n}}{\partial\alpha_1^n\dots \partial\alpha_4^n}\Bigg|_{\alpha_j=0} I_T(\alpha_1,\alpha_2,\alpha_3,\alpha_4)&=\int_T^{2T}|\zeta^{(n)}(\half+it)|^2|A(\half+it)|^2\phi(t/T)\ d t + O(T^{1-\delta}).
        \end{split}
    \end{equation*}
    The integrand in \eqref{eq:4th_moment_lemma_integral} consists of factors independent of the $\alpha_j$, and for each $1\leq j \leq 4$ a factor on the penultimate line depending on $\alpha_j$. Write this factor as
    \begin{equation*}
        \tau^{-\alpha_j/2}(P_4(\alpha_j))^{-1} = \frac{\tau^{-\alpha_j/2}}{(z_1-\alpha_j)(z_2-\alpha_j)(z_3+\alpha_j)(z_4+\alpha_j)},
    \end{equation*}
    so $P_4(\alpha_j)$ is a degree $4$ polynomial in $\alpha_j$ with coefficients depending on $z_1,\dots,z_4$. 
    
    We claim that the $n$th derivative of this with respect to $\alpha_j$ is $\ll (\log T)^{4+n}$. In fact, this derivative equals
    \begin{equation*}
        \sum_{m=0}^n \binom{n}{m}(-\half\log\tau)^{k}\frac{\partial^{n-m}}{\partial\alpha_j^{n-m}}(P_4(\alpha_j)^{-1}),
    \end{equation*}
    so it suffices to show that $\frac{\partial^{n}}{\partial\alpha_j^{n}}(P_4(\alpha_j))^{-1}\ll (\log T)^{4+n}$. A Laurent expansion of $P_4(\alpha_j)^{-1}$ at $\alpha_j=0$ then reveals that its $n$th derivative is a rational function of $z_1,\dots,z_4$ with denominator of degree $\leq n$; since $|z_1|,\dots,|z_4|\gg 1/\log T$, this proves the claim.
   
By choice of contour for the $z_j$, one also has the bounds
    \begin{align*}
        A(z_1,z_2,-z_3,-z_4)\ll \log^4 T, \quad \Delta^2(z_1,z_2,z_3,z_4)\ll (\log T)^{-12}.
    \end{align*}
    Applying Jensen's inequality to the integral expression for $\frac{\partial^{4n}}{\partial \alpha_1^n\dots\partial\alpha_4^n}I_T(\alpha_1,\dots,\alpha_4)$, noting that each contour for $z_j$ has length $\ll 1/\log T$, and collecting logarithms proves \eqref{eq:4th_moment_lemma_integral}.
\end{proof}

\subsection{Applying the moment bounds}

For $1\leq r \leq J$,  let 
\begin{equation*}
    A_r(s):=\left(\prod_{j<r}\fN_j(s,k-2)\right)\cdot\left(\frac{\fP_r(\half+it)}{50P_r}\right)^{\lceil 50 P_r \rceil},
\end{equation*} 
and let $A_{J+1}(s)=\prod_{j<J+1}\fN(s,k-2)$. 
In the proof of Proposition 3 in \cite[Section 6]{heap-radziwill-soundararajan}, it is shown for these choices of $A_r(s)$ that
\begin{align*}
    \max_{|z_j|=3^j/\log T}|G(z_1,z_2,z_3,z_4)|\ll (\log T_r)^{k^2-4}  18^r r! &P_r^r e^{P_r} \textrm{ for } 1\leq r \leq J,\\
    \max_{|z_j|=3^j/\log T}|G(z_1,z_2,z_3,z_4)|\ll (\log T)^{k^2-4}  &\textrm{ for } r = J+1
\end{align*}
respectively.
Note that each $A_r$ has length $\ll T^{1/10}$. Therefore by Lemma \ref{lem:general_twisted_4th}, for each $r$
\begin{equation}\label{eq:4th_moment_final}
    \int_T^{2T}|\zeta^{(n)}(\half+it)|^{4}|A_r(\half+it)|^2\ d t \ll T(\log T)^{4+4n}\max_{|z_j|=3^j/\log T}|G(z_1,\dots,z_4)|.
\end{equation}

To bound the moments of Dirichlet polynomials on the right-hand sides of \eqref{eq:UB_twisted_good}, \eqref{eq:UB_twisted_bad}, we have the bounds from \cite[Proposition 2]{heap-radziwill-soundararajan} (for $1\leq r \leq J$),
\begin{align}
    \int_T^{2T}\prod_{j<r}\left|\fN_j(\half+it,k)\right|^2\left|\frac{\fP_r(\half+it)}{50P_r}\right|^{2\lceil 50 P_r \rceil}\ &d t  \label{eq:G_polynomials}\ll T(\log T_r)^{k^2}(r! P^r_r),\\
    \int_T^{2T}\prod_{j<J+1}\left|\fN_j(\half+it,k)\right|^2\ d t  \label{eq:Br_polynomials}\ll & T(\log T)^{k^2}.
\end{align}
From Mertens' estimates and Stirling's approximation, for $1\leq r \leq J$,
\begin{align*}
   T(\log T_r)^{k^2}(r!P_r^r) \ll T(\log T)^{k^2}&e^{-10P_r},\\
    T(\log T)^4(\log T_r)^{k^2-4}(18^r r! P^r_re^{P_r})\ll & T(\log T)^{k^2}e^{-10P_r}.
\end{align*}
Equations \eqref{eq:zeta_UB_good} and \eqref{eq:zeta_UB_bad} now follow from substituting the moment bounds \eqref{eq:4th_moment_final}--\eqref{eq:Br_polynomials} into \eqref{eq:UB_twisted_good} and \eqref{eq:UB_twisted_bad}. This proves Theorem \ref{thm:mainRH} when $k\leq 2$.

\section{Upper Bounds: Proof of Theorem \ref{thm:mainRH} for $k\geq \half$}
\label{sec:big_h_proof}

It suffices to prove \eqref{eq:RH_Thm.1}.
The starting point of the proof is the formula
\[
\zeta^{(n)}(s) = \frac{1}{n!\cdot 2\pi i} \oint_{|z| = 1/\log T} \frac{\zeta(s+z)}{z^{n+1}} dz.
\]
Applying H\"older's inequality and then switching the order of the integrals gives
\begin{align*}
\I^{(n)}_T(k)&
\ll \int_T^{2T}\left(\oint_{|z|=1/\log T}\frac{|\zeta(\half+it+z)|}{|z|^{n+1}} \ d|z| \right)^{2k}\ dt \\
&\ll \int_T^{2T}\left(\oint_{|z|=1/\log T}|\zeta(\half+it+z)|^{2k} \ d|z|\right)(\log T)^{2nk+1}\ dt \\
&\ll (\log T)^{2nk+1} \oint_{|z|=1/\log T} \int_T^{2T}  |\zeta(\tfrac{1}{2} + i t + z)|^{2k} \ d t\\
&\ll (\log T)^{2nk} \max_{|z|= 1/\log T}  \int_T^{2T} |\zeta(\tfrac{1}{2} + i t + z)|^{2k} \ d t.
\end{align*}
This last integral may now be bounded by $\ll T(\log T)^{k^2}$ uniformly in $|z| = 1/\log T$. 
To see this, we may reduce to the case that $\Re z \geq 0$ by using the functional equation (see Lemma 22 of \cite{FHKConjecture}).
The claim now follows from combining Theorem 1 of  \cite{heap-radziwill-soundararajan} (for $k\leq 2$ unconditionally) or \cite{harper} (assuming RH) and Theorem 7.1 of  \cite{Titchmarsh} with Theorem 2 of \cite{Gabriel}.
While Theorem 2 of Gabriel's paper \cite{Gabriel} only holds for functions analytic in a strip, we may still use his result by approximating the indicator function of the rectangle $1/2 \leq \sigma \leq 3/2$, $T \leq t \leq 2T$ by a suitable analytic function, see Lemma 3.5 of \cite{arguin-ouimet-radziwill}.
Alternatively, one may show this follows by a straightforward modification of the argument in \cite{heap-radziwill-soundararajan} for $0 \leq k\leq 2$ unconditionally, or \cite{harper} for all $k\geq 0$ assuming RH.
This concludes the proof of Theorem \ref{thm:mainRH}. \qed


\end{document}